\documentclass[12pt, reqno]{amsart}
\usepackage{hyperref,xcolor,amssymb}
\usepackage{a4wide}
\usepackage[T1]{fontenc}
\usepackage[utf8]{inputenc}
\usepackage{cite}
\usepackage[numbers,sort&compress]{natbib}
\usepackage{amsmath}
\usepackage{booktabs}
\usepackage{geometry}
\usepackage{bigints}
\newtheorem{thm}{Theorem}[section]

\newtheorem{lemma}[thm]{Lemma}

\theoremstyle{definition}

\newtheorem{remark}[thm]{Remark}
\newtheorem{example}[thm]{Example}
\numberwithin{equation}{section}

\title[New inequalities related to sums of $L^p$ functions]{New inequalities related to sums of $L^p$ functions in connection with Carbery's problems}
\author{A. Aghajani}
\address[A. Aghajani]{School of Mathematics \& computer science, Iran University of Science and Technology, Narmak, Tehran, Iran }
\email{aghajani@iust.ac.ir}

\author{J. Kinnunen}
\address[J. Kinnunen]{Department of Mathematics, Aalto University, P.O. Box 11100, FI-00076 Aalto, Finland}
\email{juha.k.kinnunen@aalto.fi}
\begin{document}



\maketitle

\begin{abstract}
Carbery (2006) proposed novel estimates for the $L^p$ norm of a sum of two nonnegative measurable functions.
Subsequently, Carlen, Frank, Ivanisvili and Lieb (2018) provided stronger bounds, which Ivanisvili and Mooney (2020) further refined to achieve estimates that are, in a certain sense, optimal. 
Continuing this line of research, the present work establishes new upper and lower bounds for the range \(p\in(1,\infty)\).
Carbery also asked under what conditions on a sequence  \((f_j)\) of nonnegative measurable functions the inequality \(\sum \|f_j\|_p^p < \infty\) implies that \(\sum f_j \in L^p\). 
Ivanisvili and Mooney (2020) resolved this question for \(p\in[1,2]\), and the present work proposes an answer for \(p\in[2,\infty)\). 
\end{abstract}

\section{Introduction}

A standard way to estimate the $L^p$ norm of a sum of two functions in $L^p$ on any measure space is to apply Minkowski's inequality 
\[
\|f + g\|_p \leq \|f\|_p + \|g\|_p,
\quad p\in[1,\infty),
\]
or the bound
\begin{equation}\label{1}
\|f + g\|_p^p \leq 2^{p-1}\bigl(\|f\|_p^p + \|g\|_p^p\bigr),
\quad p\in[1,\infty).
\end{equation}
Equality occurs in (\ref{1}) if and only if \( f = g \). 
In what follows we consider functions $f, g\in L^p$, $p\in(1,\infty)$,  that are measurable and nonnegative. 
We  further assume  that,  $\|f\|_p\neq 0$ and $ \|g\|_p\neq 0$, to simplify the presentation and avoid trivial cases.  
Motivated by the inequality
\[
\|f + g\|_2^2 \leq \|f\|_2^2 + 2\|fg\|_1 + \|g\|_2^2,
\]
Carbery \cite{Carbery} proposed several refinements of (\ref{1}) including
\begin{equation}\label{Carbery}
\|f + g\|_p^p 
\leq \left(1 + \frac{\|fg\|_{p/2}}{\|f\|_p \|g\|_p}\right)^{p-1}\bigl(\|f\|_p^p + \|g\|_p^p\bigr),
\quad p\in[2,\infty).
\end{equation}
Equality occurs in (\ref{Carbery})  both if  \( f = g \) and if  $fg=0$, i.e., $f$ and $g$ have disjoint supports. 
By the Cauchy--Schwarz inequality 
\[
\frac{\|fg\|_{p/2}}{\|f\|_p \|g\|_p}\le1,
\quad p\in[2,\infty),
\]  
from which it follows that (\ref{Carbery}) improves on the trivial bound in (\ref{1}).

Carlen, Frank, Ivanisvili and Lieb \cite{Lieb1} proved that
\begin{equation}\label{Lieb}
\|f + g\|_p^p \leq \bigl(1 + \Gamma_p^{2/p})^{p-1}(\|f\|_p^p + \|g\|_p^p\bigr),
\quad  p \in (0, 1] \cup [2, \infty)
\end{equation}
where
\[
\Gamma_p=\Gamma_p(f,g) = \frac{2\|fg\|^{p/2}_{p/2}}{\|f\|_p^p + \|g\|_p^p}.
\]
A reverse inequality holds true if $ p \in (-\infty, 0) \cup (1, 2)$, with the additional assumption that
$f$ and $g$ are positive almost everywhere.
Here 
\[
\|f\|_p=\biggl(\int|f|^p\biggr)^{1/p},
\quad p\ne0.
\]
Not only  (\ref{Lieb}) is stronger than (\ref{Carbery}), it is also optimal, in the sense that $\Gamma_p^{2/p}$ cannot be
replaced by $\Gamma_p^{r}$ for any $r < 2/p$.  Later, Carlen, Frank and Lieb \cite{Lieb-ark}  successfully extended inequality (\ref{Lieb}) for sums of a finite number of functions in $L^p$, see Theorem 1.1 in \cite{Lieb-ark}. 

Ivanisvili and Mooney \cite{Mooney} refined (\ref{Lieb}) by showing that,
for any nonnegative measurable functions $f$ and $g$, we have
\begin{equation}\label{Mooney}
\|f + g\|_p^p \leq \Biggl( \left( \frac{1 + \sqrt{1 - \Gamma_p^2}}{2} \right)^{1/p} + \left( \frac{1 - \sqrt{1 - \Gamma_p^2}}{2} \right)^{1/p} \Biggr)^p
\bigl(\|f\|_p^p + \|g\|_p^p\bigr),
\end{equation}
where \( p \in (0, 1] \cup [2, \infty) \).
The inequality reverses if \( p \in (-\infty, 0) \cup [1, 2] \). Equality holds if 
$(fg)^{p/2} = \alpha(f^p + g^p)$
for some constant \(\alpha \in [0,\frac12] \).
The  function of $\Gamma_p$ on the right-hand side of (\ref{Mooney}) is shown to be optimal in \cite{Mooney} in the sense that for all $\gamma \in [0, 1]$, 
there exist functions $f$ and $g$ for which $\Gamma_p(f, g) = \gamma$, and such that equality holds in (\ref{Mooney}).

The first main result of this paper establishes new upper and lower bounds for $\|f+g\|_p^p$ for all $p\in(1,\infty)$. 
For nonnegative functions $f,g\in L^p$, with $\|f\|_p\ne0$ and $\|g\|_p\ne0$, let
\[ \overline{J}_p (f,g)= 1+\frac{ \Bigl(1+\frac{\| fg^{\frac{1}{p-1}}\|_{p-1}}{\|g^q\|_{p-1}} \Bigr)^p-1-\Bigl(\frac{\| fg^{\frac{1}{p-1}}\|_{p-1}}{\|g^q\|_{p-1}}\Bigr)^p}{1+ \Bigl(\frac{\|f\|_p}{\|g\|_p}\Bigr)^p } \]
and
\[ \underline{J}_p(f,g)= 1+ \frac{  \Big(1+\frac{\|fg^{p-1}\|_1}{\|g\|_p^p} \Big)^p- 1-\Bigl(\frac{\|fg^{p-1}\|_1}{\|g\|_p^p}\Bigr)^p}{1+\Bigl(\frac{\|f\|_p}{\|g\|_p}\Bigr)^p},\]
where $q=\frac{p}{p-1}$.

\begin{thm}\label{main} Let $ p\in[2,\infty)$. For any nonnegative functions $f,g\in L^p$, with $\|f\|_p\ne0$ and $\|g\|_p\ne0$, we have
\begin{equation}\label{main-ineq}
 \max\{\underline{J}_p (f,g),  \underline{J}_p (g,f)\} 
 \le  \frac{ \|f+g\|_p^p}{\|f\|_p^p +\|g\|_p^p}\le \min\{ \overline{J}_p (f,g), \overline{J}_p (g,f)\}.
\end{equation}
Equalities hold if, for some $\alpha,\beta>0$,  we have $\alpha f=\beta g$ almost everywhere either on the set $\{g\neq 0\}$ or on $\{f\neq 0\}$, or if $f$ and $g$ have disjoint supports up to a set of measure zero.
The reverse inequalities hold for $p\in(1,2]$, with  \( \max \)  and \( \min \) interchanged, assuming   \( f \) and \( g \) are strictly positive almost everywhere in the upper bound.
 In this case, equalities hold if $\alpha f=\beta g $ almost everywhere for some $\alpha,\beta>0$, or if $f$ and $g$ have disjoint supports up to a set of measure zero.
\end{thm}

We note that (\ref{main-ineq}) and (\ref{Mooney}) are not generally comparable. 
However, as we will see, for certain functions $f$ and $g$, the presented bounds offer an improvement over (\ref{Mooney}).

In~\cite{Carbery}, Carbery also asked under what conditions on a sequence $(f_j)$ of functions in $L^p$ the inequality $\sum \|f_j\|_p^p < \infty$ would imply that $\sum f_j \in L^p$. When attempting to extend inequality (\ref{1}) from two functions to $n$ functions $f_1, f_2, \ldots, f_n$, the constant $2^{p-1}$ is replaced by $n^{p-1}$, which blows up with $n$. 
Ivanisvili and Mooney \cite{Mooney} provided an answer to this question by showing that for a sequence $(f_j)$ of nonnegative measurable functions
\begin{equation}\label{sum-Mooney}
\Bigl\|\sum_j f_j\Bigr\|_p^p\le \sum_j \|f_j\|_p^p+(2^p-2)\sum_{i<j}  \|f_if_j\|_{p/2}^{p/2},
\quad p\in[1,2],
\end{equation}
and the inequality reverses if $p\in  (0, 1] \cup [2, \infty)$.
In particular, if $p \in [1, 2]$, then $\sum_j \|f_j\|_p^p < \infty$ and $\sum_{i<j}\|f_j f_j\|_{p/2}^{p/2} < \infty$ impy that $\sum_jf_j \in L^p$.
The second main results of the present work proposes an answer to this question for all $p\in[2,\infty)$.
\begin{thm}\label{T-many}
Let $(f_j)$ be a sequence of nonnegative functions $f_j\in L^p$, $j\in\mathbb N$.
\begin{itemize}
 \item[(i)] If $p\in[2,\infty)$, then
\begin{equation}\label{sum-u}
\Bigl\|\sum_{j=1}^n f_j\Bigr\|_p^p
\le \sum_{j=1}^n \|f_j\|_p^p+C_{p}\Bigl\|\sum_{j=1}^n f_j\Bigr\|_p^{p-2}
\Bigl\|\sum_{1\le i<j\le n}  f_if_j\Bigr\|_{p/2},\quad n\in\mathbb{N},
\end{equation}
where 
\begin{equation}\label{cp}
C_{p}=\max\left\{p,\frac{2^p-2}{2^{p-2}}\right\}.
\end{equation}
\item[(ii)]
If $p\in[1,2]$, then
\begin{equation}\label{sum-l}
\Bigl\|\sum_{j=1}^n f_j\Bigr\|_p^p
\ge \sum_{j=1}^n \|f_j\|_p^p
+\frac{2^p-2}{2^{p-2}}\frac{\bigl\|\sum_{1\le i<j\le n}  f_if_j\bigr\|_{p/2}}{\bigl\|\sum_{j=1}^n f_j\bigr\|_p^{2-p}},\quad n\in\mathbb{N},
\end{equation}
provided that $\sum_{j=1}^nf_j\neq0$ almost eveerywhere.
\end{itemize}
\end{thm}

Combining Theorem \ref{T-many} with (\ref{sum-Mooney}) we find that the following are necessary and sufficient conditions for the sum $\sum_j f_j$ to belong to $L^p$.
\begin{itemize}
\item[(i)] \textbf{Case} $p\in[1,2]$:
\begin{itemize}
\item[]\textit{Sufficient condition}: $ \sum_j \|f_j\|_p^p<\infty$ and $\sum_{i<j} \| f_if_j\|_{p/2}^{p/2}<\infty$.
\item[] \textit{Necessary condition}: $ \sum_j \|f_j\|_p^p<\infty$ and $\|\sum_{i<j}  f_if_j\|_{p/2}<\infty$.
\end{itemize}
 \item[(ii)] \textbf{Case} $p\in[2,\infty)$:
 \begin{itemize}
\item[] \textit{Sufficient condition}: $ \sum_j \|f_j\|_p^p<\infty$ and $\|\sum_{i<j}  f_if_j\|_{p/2}<\infty$.
\item[]\textit{Necessary condition}: $ \sum_j \|f_j\|_p^p<\infty$ and $\sum_{i<j}  \|f_if_j\|_{p/2}^{p/2}<\infty$.
\end{itemize}
\end{itemize}

\section{Proof of Theorem \ref{main}}

We start with  the following simple lemma which will be useful in the proof of our main estimates  for $\|f+g\|_p^p$.

\begin{lemma}\label{main'} Let $f\in L^p$ and $g\in L^q$, $q=\frac{p}{p-1}$, be nonnegative functions. Then 
\[ 
\|fg\|_1
\le \|g\|_q \bigl(\|f+g^{\frac{1}{p-1}}\|_p-\| g^{\frac{1}{p-1}} \|_p\bigr)
\le \|g\|_q \|f\|_p,
\quad p\in(1,\infty). 
\]
The reverse inequalities hold if $p\in(0,1)$, provided that $g$ is strictly positive almost everywhere.
 \end{lemma}

\begin{proof}
We only prove the case $p\in(1,\infty)$, since the case $p\in(0,1)$ is similar by using the corresponding versions of H\"older's and Minkowski's inequalities.
By H\"older's inequality, we have
\[
\int fg+\int g^q
=\int g(f+g^{q-1})\le\biggl(\int g^q\biggr)^{1/q}\biggl(\int(f+g^{\frac{1}{p-1}})^p\biggr)^{1/p},
\]
from which it follows that
\[ 
\|fg\|_1\le \|g\|_q \bigl(\|f+g^{\frac{1}{p-1}}\|_p-\| g^{\frac{1}{p-1}} \|_p \bigr). 
\]
We also note that the right-hand side is smaller or equal to $\|g\|_q \|f\|_p$ by Minkowski's inequality. 
\end{proof} 

\begin{remark}
Let $f\in L^p$ and $g\in L^p$, $p>1$, be nonnegative functions.
Then $g^{p-1}\in L^q$, $q=\frac{q}{q-1}$, and by replacing $g$ with $g^{p-1}$ in (\ref{main'}), we obtain
\[ 
\|fg^{p-1}\|_1\le \|g\|_p^{p-1}\bigl(\|f+g\|_p-\|g\|_p\bigr).
\]
As a consequence, we have the following reverse version of  Minkowski's inequality
\[\|g\|_p +\frac{\|fg^{p-1}\|_1}{\|g\|_p^{p-1}}\le \|f+g\|_p\le \|f\|_p+\|g\|_p.\]
Again the reverse inequalities hold if $p\in(0,1)$, assuming that $g$ is strictly positive almost everywhere.
\end{remark}

\begin{proof}[Proof of Theorem \ref{main}]
We only prove the case $p\in[2,\infty)$, since the argument for $p\in(1,2]$ is similar with inequalities to the reverse direction.
For $t\ge0$ we have
\begin{align*}
\frac{d}{dt}\int (f+tg)^p&= p\int g(f+tg)^{p-1}= p\int (fg^{\frac{1}{p-1}}+tg^q)^{p-1}\\&=p \|fg^{\frac{1}{p-1}}+tg^q\|_{p-1}^{p-1}.
\end{align*}
By Lemma \ref{main'}, with $p$ replaced with $p-1>1$, we obtain 
\begin{align*}
\frac{d}{dt}\int (f+tg)^p
& \ge p\biggl(  t\|g^q\|_{p-1} +\frac{\|fg^{\frac{1}{p-1}}g^{q(p-2)}\|_1}{\|g^q\|_{p-1}^{p-2}}\biggr)^{p-1}\\
&= p\biggl(  t\|g\|^q_{p} +\frac{\|fg^{p-1}\|_1}{\|g\|_{p}^{q(p-2)}}\biggr)^{p-1}
= \frac{d}{dt} \frac{\Bigl(  t\|g\|^q_{p} +\frac{\|fg^{p-1}\|_1}{\|g\|_{p}^{q(p-2)}}\Bigr)^{p}}{\|g\|_p^q}\\&
=\frac{d}{dt} \frac{\bigl(t\|g\|_p^p+\|fg^{p-1}\|_1\bigr)^p}{\|g\|_p^{p(p-1)}}.
\end{align*}
It follows that $F'(t)\ge0$, where
\[
F(t)=\int (f+tg)^p-\frac{\bigl(t\|g\|_p^p+\|fg^{p-1}\|_1\bigr)^p}{\|g\|_p^{p(p-1)}}.
\]
Therefore, $F(1)\ge F(0)$ and, after a simplification, we arrive at
\begin{align*}
\|f+g\|_p^p
&\ge \|f\|_p^p +\frac{\bigl(\|g\|^p_p+\|fg^{p-1}\|_1\bigr)^p- \|fg^{p-1}\|_1^p  }{ \|g\|_{p}^{p(p-1)}}\\
&=\|f\|_p^p +\|g\|_p^p \biggl( \biggl(1+\frac{\|fg^{p-1}\|_1}{\|g\|_p^p} \biggr)^p-\biggl(\frac{\|fg^{p-1}\|_1}{\|g\|_p^p}\biggr)^p\biggr) \\
&= \|f\|_p^p +\|g\|_p^p +\|g\|_p^p \biggl( \biggl(1+\frac{\|fg^{p-1}\|_1}{\|g\|_p^p} \biggr)^p- 1-\biggl(\frac{\|fg^{p-1}\|_1}{\|g\|_p^p}\biggr)^p\biggr) 
\end{align*}
or equivalently,
\[ 
\frac{ \|f+g\|_p^p}{\|f\|_p^p +\|g\|_p^p}
\ge 1+ \frac{  \Bigl(1+\frac{\|fg^{p-1}\|_1}{\|g\|_p^p} \Bigr)^p- 1-\Bigl(\frac{\|fg^{p-1}\|_1}{\|g\|_p^p}\Bigr)^p}{1+\bigl(\frac{\|f\|_p}{\|g\|_p}\bigr)^p}
=\underline{J}_p(f,g).
\]
The corresponding inequality also holds by interchanging the roles of $f$ and $g$.

To prove the inequality involving $\overline{J}_p(f,g)$, for $t\ge0$ we have
\begin{align*}
\frac{d}{dt} \int (f+tg)^p
&= p\int g(f+tg)^{p-1}
= p\int (fg^{\frac{1}{p-1}}+tg^q)^{p-1}\\
&=p \|fg^{\frac{1}{p-1}}+tg^q\|_{p-1}^{p-1}. 
\end{align*}
By Minkowski's inequality we have
\begin{equation}\label{triangle}
\begin{split}
p \|fg^{\frac{1}{p-1}}+tg^q\|_{p-1}^{p-1}
&\le p\bigl(\| fg^{\frac{1}{p-1}}\|_{p-1}+t\|g^q\|_{p-1}\bigr)^{p-1} \\
&=\frac{d}{dt} \frac{\bigl(\| fg^{\frac{1}{p-1}}\|_{p-1}+t\|g^q\|_{p-1}\bigr)^{p}}{\|g^q\|_{p-1}}.
\end{split}
\end{equation}
Hence $ G'(t)\le0$, where
\[
G(t)=\|f+tg\|_p^p- \frac{\bigl(\| fg^{\frac{1}{p-1}}\|_{p-1}+t\|g^q\|_{p-1}\bigr)^{p}}{\|g^q\|_{p-1}}.
\]
Thus $G(1)\le G(0)$ and 
\begin{align*}
\|f+g\|_p^p
&\le \|f\|_p^p+ \frac{\bigl(\| fg^{\frac{1}{p-1}}\|_{p-1}+\|g^q\|_{p-1}\bigr)^{p}-\| fg^{\frac{1}{p-1}}\|_{p-1}^p}{\|g^q\|_{p-1}}\\
&= \|f\|_p^p+\|g\|_p^p\biggl(\biggl(1+\frac{\| fg^{\frac{1}{p-1}}\|_{p-1}}{\|g^q\|_{p-1}}\biggr)^p-\frac{\| fg^{\frac{1}{p-1}}\|_{p-1}^p}{\|g^q\|_{p-1}^p}\biggr)\\
&=\|f\|_p^p+\|g\|_p^p+\|g\|_p^p\biggl( \biggl(1+\frac{\| fg^{\frac{1}{p-1}}\|_{p-1}}{\|g^q\|_{p-1}} \biggr)^p-1-\biggl(\frac{\| fg^{\frac{1}{p-1}}\|_{p-1}}{\|g^q\|_{p-1}}\biggr)^p\biggr)\\
&=\Biggl(1+\frac{ \Bigl(1+\frac{\| fg^{\frac{1}{p-1}}\|_{p-1}}{\|g^q\|_{p-1}} \Bigr)^p-1-\Bigl(\frac{\| fg^{\frac{1}{p-1}}\|_{p-1}}{\|g^q\|_{p-1}}\Bigr)^p}{1+ \Bigl(\frac{\|f\|_p}{\|g\|_p}\Bigr)^p }\Biggr) 
(\|f\|_p^p+\|g\|_p^p)\\
&=\overline{J}_p (f,g)\bigl(\|f\|_p^p+\|g\|_p^p\bigr).
\end{align*}
The corresponding inequality also holds by interchanging the roles of $f$ and $g$.
For the equality part, see Remark \ref{equality}.
\end{proof}

\begin{remark} 
Let $p\in[2,\infty)$,
By H\"older's inequality we have
\begin{equation}\label{x}
\frac{\| fg^{\frac{1}{p-1}}\|_{p-1}}{\|g^q\|_{p-1}}\le \frac{\|f\|_p}{\|g\|_p}.
\end{equation}
Let 
\[
T=\frac{\|f\|_p}{\|g\|_p}
\]
and let $0\le\lambda\le 1$ be such that
\[
\lambda T=\frac{\| fg^{\frac{1}{p-1}}\|_{p-1}}{\|g^q\|_{p-1}}.
\] 
The function $H(t)=(1+t)^p-1-t^p$ is convex on $[0,\infty)$, thus
\[
H(\lambda T)\le (1-\lambda)H(0)+\lambda H(T)= \lambda H(T).
\]
By (\ref{main-ineq}) we have
\[
\begin{split}
\frac{ \|f+g\|_p^p}{\|f\|_p^p +\|g\|_p^p}
&\le \overline{J}_p(f,g)
=1+\frac{H(\lambda T)}{1+T^p}
\le1+\frac{\lambda H(T)}{1+T^p}\\
&=1+\lambda \frac{(1+T)^p-1-T^p}{1+T^p}
=(1-\lambda)+\lambda \frac{(1+T)^p}{1+T^p}\\
&=(1-\lambda)+\lambda \frac{ (\|f\|_p+\|g\|_p)^p}{\|f\|_p^p +\|g\|_p^p}.
\end{split}
\]
It follows that
\begin{equation}
\|f+g\|_p^p\le \lambda (\|f\|_p+\|g\|_p)^p+(1-\lambda) (\|f\|_p^p +\|g\|_p^p).
\end{equation}
In particular,
\[\frac{ \|f+g\|_p^p}{\|f\|_p^p +\|g\|_p^p}\le (1-\lambda)+\lambda \frac{( \|f\|_p+\|g\|_p)^p}{\|f\|_p^p +\|g\|_p^p}\le(1-\lambda)+2^{p-1}\lambda\le 2^{p-1}.\]

Note also that   we can simplify the left-hand side inequality in (\ref{main-ineq}), using the  inequality
\begin{equation}\label{mooney-ineq}
(a+b)^p\ge a^p+b^p+(2^p-2)(ab)^{p/2},
\quad p\in[2,\infty),
\end{equation} 
for nonnegative numbers 
\[
a=1
\quad\text{and}\quad 
b=\frac{\|fg^{p-1}\|_1}{\|g\|_p^p},
\] 
see \cite{Mooney}.
By (\ref{mooney-ineq}) we have
\begin{equation}\label{simple}
\frac{ \|f+g\|_p^p}{\|f\|_p^p +\|g\|_p^p}\ge  \underline{J}_p(f,g)\ge 1+ \frac{  (2^p-2) \Bigl(\frac{\|fg^{p-1}\|_1}{\|g\|_p^p}\Bigr)^{p/2}  }{1+\Bigl(\frac{\|f\|_p}{\|g\|_p}\Bigr)^p},
\quad p\in[2,\infty).
\end{equation}
Again the reverse inequalities hold for $p\in[1,2]$.

In particular, for nonnegative $f,g\in L^p$ with $\|f\|_p=1$ and $\|g\|_p=1$, from (\ref{simple}) we have
\[
\|f+g\|_p^p\ge 2+(2^p-2)\|fg^{p-1}\|_1^{p/2},
\quad p\in[2,\infty), 
\]
or equivalently
\[ 
1-\left\|\frac{f+g}{2}\right\|_p^p\le (1-2^{1-p})\bigl(1-\|fg^{p-1}\|_1^{p/2}\bigr),
\quad p\in[2,\infty), 
\]
and the inequality reverses if $p\in[1,2]$.

Moreover, by using the inequality above and Hanner's inequality
\[
\left\|\frac{f+g}{2}\right\|_p^p+\left\|\frac{f-g}{2}\right\|_p^p\le1,
\quad p\in[2,\infty), 
\]
see \cite{Hanner, Lieb-book, Lieb1} (also a consequence Clarkson's inequalities \cite{Clarkson}),
we have
\[
\|f-g\|_p^p\le (2^p-2)\bigl(1-\|fg^{p-1}\|_1^{p/2}\bigr),
\quad p\in[2,\infty).
\]
Again the reverse inequality holds for $p\in(1,2]$.

\end{remark}
\begin{remark}\label{equality}
By applying the explicit upper and lower bounds in (\ref{main-ineq}), we obtain the error estimate
\begin{equation}\label{error}
\left|\overline{J}_p(f,g) -  \underline{J}_p(f,g)\right|
\le p2^{p-1} \biggl|\frac{\| fg^{\frac{1}{p-1}}\|_{p-1}}{\|g^q\|_{p-1}}- \frac{\|fg^{p-1}\|_1}{\|g\|_p^p}\biggr|,
\quad p\in(1,\infty),
\end{equation}
whenever $\|f\|_p\le \|g\|_p$. If $\|f\|_p \ge \|g\|_p$, the roles of $f$ and $g$ are interchanged.

To prove (\ref{error}), let $H(t)=(1+t)^p-t^p$, $t\ge0$,
\[
T_1=\frac{\|fg^{p-1}\|_1}{\|g\|_p^p}
\quad\text{and}\quad
T_2= \frac{\| fg^{\frac{1}{p-1}}\|_{p-1}}{\|g^q\|_{p-1}}.
\] 
Then we have
\[
\left|\overline{J}_p(f,g) -  \underline{J}_p(f,g)\right|
=\frac{|H(T_2)-H(T_1)|}{1+ \Bigl(\frac{\|f\|_p}{\|g\|_p}\Bigr)^p}
\le |H(T_2)-H(T_1)|. 
\]
Let $p\in[2,\infty)$. Since $\frac{p-2}{p-1}+\frac{1}{p-1}=1$,  H\"older's inequality gives
\begin{equation}\label{holder}
\int fg^{p-1}
=\int (fg^{\frac{1}{p-1}}) g^{\frac{p(p-2)}{p-1}}
\le\biggl(\int f^{p-1}g\biggr)^{1/(p-1)}\biggl(\int g^p\biggr)^{1-1/(p-1)},
\end{equation}
or equvalently
\begin{equation}\label{x>y}
T_1= \frac{\|fg^{p-1}\|_1}{\|g\|_p^p}\le \frac{\| fg^{\frac{1}{p-1}}\|_{p-1}}{\|g^q\|_{p-1}}=T_2,
\end{equation}
By (\ref{x>y}), (\ref{x}) and the assumption $\|f\|_p\le \|g\|_p$ we get $0\le T_1\le T_2\le1$. 
We may assume that $0\le T_1<T_2\le1$, since $\overline{J}_p(f,g) = \underline{J}_p(f,g)$ if $T_1=T_2$.
By the mean value theorem, there exists $T_0 \in(T_1,T_2)$ such that 
\[
H(T_2)-H(T_1)=H'(T_0)(T_2-T_1).
\]
Since 
\[
H'(T_0)=p(1+T_0)^{p-1}-pT_0^{p-1}\le p2^{p-1},
\] 
we obtain \eqref{error} for $p\in[2,\infty)$. The case $p\in(0,1)$ follows similarly.

Finally, note that equality $\overline{J}_p(f,g) = \underline{J}_p(f,g)$ holds if and only if $T_1 = T_2$, i.e., when equality holds in H\"older's inequality \eqref{holder}. 
If $\|f^{p-1}g\|_1\ne0$ and $\|g\|_p\ne0$, this occurs if and only if there exist constants $\alpha, \beta>0$ such that $\alpha f^{p-1} g = \beta g^p$ almost everywhere,
or equivalently $\alpha f = \beta g$ almost everywhere in $\{ g \neq 0 \}$.
On the other hand, if $\|f^{p-1}g\|_1=0$, then $fg=0$ almost everywhere and $f$ and $g$ have disjoint supports up to a set of measure zero.
\end{remark}

\begin{example}\label{comparison}  
We discuss functions for which the upper bound in~(\ref{main-ineq}) is stonger than the upper bound in~(\ref{Mooney}). 
Consider measurable sets $A$ and $B$ of strictly positive measure and \( A \cap B = \emptyset \). For any \(\alpha,\beta > 0 \), let
\[
f = \alpha\phi\chi_{A} +\psi_1 \chi_{B} 
\quad\text{and}\quad
g =\beta\phi \chi_{A}+\psi_2\chi_{B},
\]
where $\phi\in L^p(A)$, $\phi>0$ and $\psi_1,\psi_2\in L^p(B)$, $\psi_1,\psi_2\ge0$. 
First assume that $\psi_2\equiv 0$. 
In this case we have $\beta f=\alpha g$ on $\{g\neq 0\}$, by Remark \ref{equality} we have $\overline{J}_p(f,g) = \underline{J}_p(f,g)$.
Thus we have equality in both sides of (\ref{main-ineq}), whereas this is not the case for the inequality in~(\ref{Mooney}), when $p\neq 2$. 
This also demonstrates that, at least when $\|\psi_2\|_{L^p(\Omega_2)}>0$ is sufficiently small, the upper bound in~(\ref{main-ineq}) provides a strictly smaller than the upper bound in~(\ref{Mooney}).
\end{example}

\section{Proof of Theorem \ref{T-many}}

We start with an elementary inequality for real numbers.
\begin{lemma}\label{many}
Let $C_p$ as in (\ref{cp}).
\begin{itemize}
 \item[(i)] Let $p\in[2,\infty)$. For nonnegative real numbers $a$ and $b$, we have
\begin{equation}\label{ab}
(a+b)^p\le a^p+b^p+C_{p}(a+b)^{p-2}ab.
\end{equation}
More generally, for any sequence $(a_i)$ of nonnegative real numbers, we have 
\begin{equation}\label{ai}
\Bigl(\sum_{j=1}^na_j\Bigr)^p
\le \sum_{j=1}^na_j^p +C_{p}\Bigl(\sum_{j=1}^na_j\Bigr)^{p-2} \sum_{1\le i<j\le n}  a_ia_j\quad n\in\mathbb{N}.
\end{equation}
 \item[(ii)] Let  $p\in[1,2]$. For nonnegative real numbers $a$ and $b$, we have
\begin{equation}\label{ab2}
(a+b)^p\ge a^p+b^p+\left(\frac{2^p-2}{2^{p-2}}\right)\frac{ab}{ (a+b)^{2-p}}.
\end{equation}
More generally, for any sequence $(a_i)$ of nonnegative real numbers, we have 
\begin{equation}\label{ai2}
\Bigl(\sum_{j=1}^na_j\Bigr)^p
\ge \sum_{j=1}^na_j^p 
+\left(\frac{2^p-2}{2^{p-2}}\right)\frac{ \sum_{1\le i<j\le n}  a_ia_j}{\bigl(\sum_{j=1}^na_j\bigr)^{2-p}},
\quad n\in\mathbb{N},
\end{equation}
provided that $\sum_{j=1}^na_j\neq0$.
\end{itemize}
\end{lemma}

\begin{proof}
We first prove (\ref{ab}). The case $ab=0$ is obvious. Thus we may assume that $a,b>0$. 
Without loss of generality we may assume that $a\ge b$. 
Let $t=\frac{b}{a}\le 1$. With this notation (\ref{ab}) is equivalent to
\[(1+t)^p\le 1+t^p+ C_{p} t(1+t)^{p-2}.\]
This inequality holds true, since for $p\in[2,\infty)$ we have
\[
\sup_{0<t\le 1} \frac{(1+t)^p-1-t^p}{ t(1+t)^{p-2}}=\max\left\{p,\frac{2^p-2}{2^{p-2}}\right\}=C_p.
\]

Next we discuss (\ref{ai}). We only show (\ref{ai}) for $n=3$, and the general case follows by induction. For  $a,b,c\ge0$, using (\ref{ab}) we obtain
\begin{align*}
(a+b+c)^p&\le (a+b)^p+c^p+C_{p}(a+b+c)^{p-2}((a+b)c)\\
&\le  a^p+b^p+C_{p}(a+b)^{p-2}ab+c^p+C_{p}(a+b+c)^{p-2}(ac+bc)\\
&\le a^p+b^p+c^p+C_{p}(a+b+c)^{p-2}(ab+ ac+bc).
\end{align*}

Then we discuss (\ref{ab2}). The case $ab=0$ is obvious. 
Without loss of generality, we may assume that $a\ge b>0$. Let $t=\frac{b}{a}\le 1$.
With this notation (\ref{ab2}) is equivalent to
\[(1+t)^p\ge 1+t^p+\left(\frac{2^p-2}{2^{p-2}}\right) t(1+t)^{p-2}.\]
This inequality holds true, since for $p\in[1,2]$ we have
\[\inf_{0<t\le 1} \frac{(1+t)^p-1-t^p}{t(1+t)^{p-2}}
=\frac{2^p-2}{2^{p-2}}
=C'_p .
\]

Finally we consider (\ref{ai2}). We only prove (\ref{ai2}) for $n=3$, and the general case follows by induction. 
For $a,b,c\ge0$, using (\ref{ab2}) we have
\begin{align*}
(a+b+c)^p&\ge (a+b)^p+c^p+C'_{p}\frac{(a+b)c}{(a+b+c)^{2-p}}\\&
\ge  a^p+b^p+C'_{p}\frac{ab}{ (a+b)^{2-p}}+c^p+C'_{p}\frac{ac+bc}{(a+b+c)^{2-p}}\\&
\ge  a^p+b^p+c^p+C'_{p}\frac{ab+ac+bc}{(a+b+c)^{2-p}}.
\end{align*}
\end{proof}

\begin{proof}[Proof of Theorem \ref{T-many}]
Let $p\in[2,\infty)$. We apply (\ref{ai}) with $a_j = f_j(x)$ and integrate the obtained inequality to get
\begin{equation}\label{sn}
\Bigl\|\sum_{j=1}^n f_j\Bigr\|_p^p
=\int\Bigl(\sum_{j=1}^nf_j\Bigr)^p
\le \int \sum_{j=1}^n f_j^p +C_{p} \int\Bigl(\Bigl(\sum_{j=1}^nf_j\Bigr)^{p-2} \sum_{1\le i<j\le n} f_if_j\Bigr).
\end{equation}
Since $\frac{p-2}{p}+\frac{2}{p}=1$, H\"older's inequality implies that
\begin{align*}
\int \Bigl(\Bigl(\sum_{j=1}^nf_j\Bigr)^{p-2} \sum_{1\le i<j\le n} f_if_j\Bigr) 
&\le   
\biggl(\int\Bigl(\sum_{j=1}^nf_j\Bigr)^p\biggr)^{(p-2)/p}
\biggl(\int\Bigl(\sum_{1\le i<j\le n}  f_if_j\Bigr)^{p/2}\biggr)^{2/p}.
\end{align*}
Inserting this inequality into (\ref{sn}) gives
\[
\Bigl\|\sum_{j=1}^n f_j\Bigr\|_p^p
\le  \sum_{j=1}^n\| f_j\|_p^p 
+C_{p}\Bigl\|\sum_{j=1}^n f_j\Bigr\|_p^{p-2}\Bigl\|\sum_{1\le i<j\le n}  f_if_j\Bigr\|_{p/2}.
\]
This proves (\ref{sum-u}). 
A similar argument proves (\ref{sum-l}) for  $p\in[1,2)$, by using the corresponding version of H\"older's inequality.
\end{proof}

\section{Acknowledgments}
 This work was carried out
during the first author’s visit at the Department of Mathematics at Aalto University. He
would like to thank the institution and the Nonlinear Partial Differential Equations group for
the kind and warm hospitality

\end{document}